\RequirePackage{fix-cm}
\documentclass{cocv}
\usepackage{amsmath,amsthm}
\usepackage{mathtools}
\usepackage{xcolor}
\usepackage{amssymb}
\usepackage{amsfonts}
\usepackage{graphicx}
\usepackage{booktabs}
\usepackage{units}
\usepackage{comment}
\usepackage{caption}
\usepackage{subcaption} % for side-by-side images
\usepackage{dblfloatfix}   % patches LaTeX’s double-column float rules
\usepackage{algorithm}
\usepackage{algorithmicx}
\usepackage{algpseudocode}
\usepackage{enumerate}

\algrenewcommand\algorithmicindent{0.5em}
\usepackage{commath}

\usepackage{hyperref}
\hypersetup{
    colorlinks=true,
    linkcolor=blue,     
    urlcolor=blue,
}
\usepackage{xparse}

\usepackage{tikz}
\usepackage{pgfplots}
\pgfplotsset{compat=1.18}
\usepackage{appendix}

\renewcommand{\d}{\,\mathrm{d}}
\renewcommand{\Re}{\mathrm{Re}}
\newcommand{\ds}{ {\, \rm d} s}
\newcommand{\nm}[2][]{\left\lVert #2 \right\rVert_{#1}}
\newcommand{\inprod}[2]{\left\langle#1,#2\right\rangle}

\newcommand{\RR}{\mathbb R}
\newcommand{\CC}{\mathbb C}

\DeclareMathOperator{\loc}{loc}

\newcommand{\B}{\mathbb{B}}
\floatname{algorithm}{Pseudo code}

\newtheorem{definition}{Definition}

\allowdisplaybreaks
\newcounter{thmcounter}[section]  % Resets with each new section
\renewcommand{\thethmcounter}{\thesection.\arabic{thmcounter}}

\newcommand{\defthmwithqed}[2]{%
  \NewDocumentEnvironment{#1}{ o }{%
    \refstepcounter{thmcounter}% increment counter and allow referencing
    \begin{trivlist}%
      \item[\hskip \labelsep \bfseries #2~\thethmcounter%
        \IfValueT{##1}{\ (##1)}.]%
  }{%
    \hfill $\square$%
    \end{trivlist}%
  }%
}

\theoremstyle{definition}
\newtheorem{theorem}{Theorem}
\newtheorem{lemma}{Lemma}

\newtheorem{proposition}{Proposition}

\newtheorem{remark}{Remark}

\usetikzlibrary{arrows.meta,positioning,calc,decorations.pathmorphing}

{\left(\begin{smallmatrix}}%            begin code
{\end{smallmatrix}\right)}%             end code

\newenvironment{smallbmatrix}%          environment name
{\left[\begin{smallmatrix}}%            begin code
{\end{smallmatrix}\right]}%       

\begin{document}
%%-----------------------------
%%      the top matter
%%-----------------------------
\title{Well-posedness and passivity for a class of bilinear control systems}\thanks{This work was funded by the Deutsche Forschungsgemeinschaft (DFG, German
Research Foundation) – Project-ID 531152215 – CRC~1701. We thank Daniel Happ (University of Wuppertal) who pointed out a possible simplification in the proof of Theorem~\ref{thm:classical_solutions}.}% At most 5 thanks
\author{Abdelhakim Dahmani}\address{Port-Hamiltonian Institute, Bergische Universität Wuppertal, Gaußstraße 20, 42119 Wuppertal, Germany}
\author{Hannes Gernandt}\address{Port-Hamiltonian Institute, Bergische Universität Wuppertal, Gaußstraße 20, 42119 Wuppertal, Germany and Fraunhofer-Einrichtung für Energieinfrastrukturen und Geotechnologien (IEG), Gulbener Straße 23, 03046 Cottbus, Germany}
\author{René Hosfeld}\address{TU Berlin, Fasanenstraße 89,
10623 Berlin, Germany}
\author{Timo Reis}\address{TU Ilmenau, Weimarer Straße 25, 98693 Ilmenau, Germany}
\author{Tolgahan Tasci}\address{Port-Hamiltonian Institute, Bergische Universität Wuppertal, Gaußstraße 20, 42119 Wuppertal, Germany}
\date{\today}
\begin{abstract} 
We study an abstract class of bilinear infinite-dimensional systems that arises
in various applications, including district heating systems and quantum control.
First, we analyze the existence of mild and classical solutions for abstract
bilinear systems with admissible input operators in Banach spaces and study
continuous dependence on the data. When the underlying space is a Hilbert space,
these results are used to show passivity for a suitably defined co-located output
and for mild solutions. The results are applied to the bilinear Schr\"odinger
equation, the Fokker--Planck equation, and a district heating or cooling pipe.
\end{abstract}
%
%\begin{resume} \end{resume}
%
\maketitle
%%-----------------------------
%%      your text
%%-----------------------------
\section{Introduction}

Bilinear control systems form an important class of systems that arise naturally in many applications. They occur, for instance, in power systems~\cite{GerSZMS25}, district heating systems~\cite{vdHeFRSSBMN17}, diffusion processes~\cite{Breiten18}, and quantum control systems \cite{Beauchard2013}. A typical bilinear system is of the form
\[
    \dot x(t) = A x(t) + \sum_{j=1}^m u_j(t) B_j x(t) + B_0 u_0(t),
\]
for some linear operators $A$ and $B_0,B_1,\ldots,B_m$, where the control $u=(u_0,\ldots,u_m)$ enters both additively and through products with the state.

From a control-theoretic point of view, bilinear systems are challenging because they occupy an intermediate position between linear and genuinely nonlinear systems. On the one hand, the dependence on the state is linear for each fixed input. On the other hand, the interaction between the state and the control introduces a nonlinear input-state relation. In particular, for a fixed control input one obtains a linear non-autonomous system whose generator depends on time through the control. Thus, methods from linear system theory cannot be applied directly, while general nonlinear techniques often fail to exploit the additional structure provided by bilinearity.

A central concept in the analysis and control of dynamical systems is passivity. Passivity expresses that the energy stored in the system cannot increase faster than the power supplied through the input and output ports. It is closely related to dissipativity and plays an important role in stability analysis, feedback design, and optimal control. For finite-dimensional bilinear systems, passivity properties are well studied, and the bilinear structure can be exploited, for example, in the construction of stabilizing feedback laws; see, for instance, \cite{Monafred23}. Passivity-based stabilization has also been investigated for infinite-dimensional systems in \cite{Berrahmoune2010}. Moreover, passivity has been used in optimal control for finite-dimensional linear systems \cite{Schaller21} and for infinite-dimensional systems \cite{hastir2026}.

For infinite-dimensional systems, passivity is particularly important because many physically motivated models are naturally described by energy balances. A prominent example is the class of port-Hamiltonian systems, where the Hamiltonian represents the stored energy and the port variables describe the exchange of energy with the environment. The infinite-dimensional theory has been developed primarily for linear systems, including systems with boundary control and observation. Passivity properties of time-varying systems have recently been analyzed in several settings, for instance for infinite-dimensional systems in \cite{SchnaubeltWeiss10}, for systems with port-Hamiltonian structure in \cite{Kurula24}, and for boundary control systems in \cite{JacoLaas21}. Further passivity results are known for certain classes of well-posed linear systems, for example in the framework of system nodes and passive systems \cite{Staffans2002,Philipp2025}.

Nevertheless, the passivity theory for infinite-dimensional bilinear systems is still less developed. This is due to several difficulties, including the possible unboundedness of the generator, the occurrence of unbounded control and observation operators, for instance in boundary control and observation, and the proper definition of a co-located output. In particular, even the well-posedness of the associated non-autonomous Cauchy problem requires careful analysis. By well-posedness we mean existence, uniqueness, and continuous dependence of solutions (and outputs) on the initial data and on the input. Several related well-posedness results are available in the literature. In \cite{Kato70}, abstract non-autonomous hyperbolic evolution equations are studied, and existence of classical solutions is proved under the assumption that the control functions are continuous for the linear part and continuously differentiable for the bilinear part. These results are also extended to larger classes of control functions in the purely bilinear case, where the linear part is absent. However, the theory is considerably richer in the parabolic case; classical references include \cite{Amann95} and \cite{lunardi1995analytic}. Furthermore, we refer to the survey \cite{Schnaubelt2002}, which contains a discussion about the well-posedness of both classes as well as their asymptotic behavior. Much less is known when the bilinear control operators are represented by unbounded operators $B_{1},\dots,B_{m}$. Nevertheless, several results for parabolic problems have been obtained within the framework of maximal $L^{p}$-regularity, where existence is studied in a different functional-analytic setting; see, for example, \cite{Arendt07} for details. Semilinear non-autonomous boundary control systems are considered in \cite{SchmDJL19}. Moreover, in \cite{AroBonKro2018}, weak and mild solutions for bilinear systems are analyzed in the case of one bilinear control term and with bounded control and observation operators; the authors also study related optimal control problems. These contributions provide important foundations, but they do not yield a passivity theory for bilinear systems with possibly unbounded admissible control operators.

The connection between bilinear systems and port-Hamiltonian systems is also of interest. In the infinite-dimensional port-Hamiltonian literature, the linear case is the most thoroughly understood. Some modeling frameworks allow nonlinearities in the Hamiltonian, but solvability and continuous dependence are not always addressed in detail. Recent contributions on nonlinear infinite-dimensional monotone port-Hamiltonian systems \cite{GernScha25} and on system-node formulations \cite{Philipp2025} provide new tools for treating nonlinear and time-varying effects. However, the specific bilinear situation, where the control acts multiplicatively on the state, still requires a dedicated analysis.

The purpose of this paper is to contribute to this theory. Our first main contribution is an existence and continuous-dependence result for abstract bilinear systems on Banach spaces given by 
\begin{equation}\label{eq:bilinear_intro}
    \tag{$\Sigma$} \left\{
        \begin{aligned}
            \dot x(t) &= Ax(t) + B_1F(u_1(t),x(t)) + B_2u_2(t), \quad t\in (0, \infty),\\
            x(0) &= x_0 \in X,
        \end{aligned} \right.
    \end{equation}
where $A$ is the generator of a $C_0$-semigroup on a~Banach space $X$,
$X_{-1}$ denotes the associated extrapolation space,
$B_1\in\mathcal L(Z,X_{-1})$ and $B_2\in\mathcal L(U_2,X_{-1})$ are
admissible control operators, and $F\colon U_1\times X\to Z$ is bounded and
bilinear. Our examples show that this framework also covers certain control
operators that are unbounded as operators into $X$.

This result provides a well-posedness framework for an abstract class of non-autonomous evolution equations with bilinear control terms. The Banach-space setting is useful because it covers many evolution equations that do not naturally fit into a Hilbert space framework.
Our second main contribution concerns passivity. Using the well-posedness result, we prove passivity properties for a class of bilinear systems with admissible control operators. For this part, we work in Hilbert spaces, where the energy identity and the duality between input and output spaces can be formulated in a natural way. This restriction is also consistent with most of the existing literature on passive infinite-dimensional systems. Passivity in Banach spaces is more delicate, and only partial results are currently available; see, for example, \cite{Iftime05,Reis2021}.

The results of this paper show that the bilinear structure can be incorporated into the passivity framework without losing well-posedness, provided suitable admissibility and energy-balance assumptions are imposed. In particular, the theory developed here gives a basis for passivity-based feedback design and stabilization of infinite-dimensional bilinear systems. It also complements existing results on finite-dimensional bilinear systems, time-varying passive systems, and port-Hamiltonian systems with boundary control.
The paper is organized as follows. In Section~\ref{sec:well-posedness} we show the well-posedness result for bilinear systems on Banach spaces. In Section~\ref{sec:passive}, we specialize to Hilbert spaces and prove the main passivity theorem for systems with admissible control operators. In Section~\ref{sec:applications} we apply the passivity results to several examples including a bilinear Schrödinger equation, a~Fokker--Planck equation and a district heating pipe. Finally, Section~\ref{sec:conclusion} summarizes the results and indicates possible directions for future research.

\subsection*{Notation}

The space of bounded linear operators $T:X\rightarrow Y$ between two normed spaces $X$ and $Y$ is denoted by $\mathcal{L}(X,Y)$ and if $X=Y$ we abbreviate this set by $\mathcal{L}(X)$. For a linear operator $T$, we denote its domain by $D(T)$ and its resolvent set by $\rho(T)$. The topological dual space of $X$ is denoted by $X^*$ and the dual pairing between $X^*$ and $X$ is $\langle \cdot, \cdot \rangle_{X^*,X}$. If $X$ is a Hilbert space, the dual pairing can be identified with the inner product $\langle \cdot , \cdot \rangle_X$ via the Riesz isomorphism. 
For an interval $I\subseteq \RR$ and $p \in [1,\infty)$ we denote the standard Lebesgue and Sobolev spaces by
\begin{align*}
L^p(I;X) &\coloneqq \left\{f: I \to X \text{ measurable }\middle|\, \int_I  \|f(t)\|_X^p \d t < \infty \right\},\\
L^p_{\loc}(I;X) &\coloneqq \left\{f: I \to X \text{ measurable } \middle|\, \int_K \|f(t)\|_X^p \d t < \infty \text{ for all compact }K \subseteq I\right\},\\
W^{1,p}_{\loc}(I;X) &\coloneqq \left\{f \in L^p_{\loc}(I;X) \mid f \text{ has weak derivative } \dot f \in L^p_{\loc}(I;X) \right\}.
\end{align*}
Furthermore, for $p=2$ we write $H^1(I;X)$ and $H_{\loc}^1(I;X)$ instead of $W^{1,2}(I;X)$ and $W^{1,2}_{\loc}(I;X)$, respectively. 
Moreover, $C(I;X)$ is the Banach space of all continuous $X$-valued functions endowed with the supremum norm 
   $ \nm[{C(I;X)}]{f}= \sup_{t\in I} \nm[X]{f(t)}$ 
and $C^1(I;X)$ is the subspace of all continuously differentiable $X$-valued functions. If $X\in \{\RR,\CC\}$ we omit the dependence on $X$ in all of the spaces introduced above.

Let $A$ be the generator of a $C_0$-semigroup $(T(t))_{t \geq 0}$ on $X$. The extrapolation space $X_{-1}$ associated with $A$ is defined as the completion of $X$ with respect to the norm $\norm{ \cdot }_{X_{-1}} \coloneqq \norm{(\lambda - A)^{-1} \,\cdot\, }_X$ for some $\lambda \in \rho(A)$. The resolvent identity yields that different choices of $\lambda \in \rho(A)$ give rise to equivalent norms. It is well-known that $X$ is continuously and densely embedded into $X_{-1}$, $(T(t))_{t \geq 0}$ uniquely extends to a $C_0$-semigroup $(T_{-1}(t))_{t \geq 0}$ on $X_{-1}$ whose generator $A_{-1}$ with $D(A_{-1})=X$ is the extension of $A$ in $X_{-1}$.

\section{Solution theory and continuous dependence on data}
\label{sec:well-posedness}

Consider the following abstract bilinear control system
\begin{equation}\label{eq:bilinear}
    \tag{$\Sigma$} \left\{
        \begin{aligned}
            \dot x(t) &= Ax(t) + B_1F(u_1(t),x(t)) + B_2u_2(t), \quad t\in (0, \infty),\\
            x(0) &= x_0 \in X,
        \end{aligned} \right.
    \end{equation}
 where $X$, $U_1$, $U_2$ and $Z$ are Banach spaces, $A$ is the generator of a $C_0$-semigroup $(T(t))_{t\geq 0}$ on $X$, and the control operators $B_1, B_2$ satisfy $B_1 \in \mathcal{L}(Z,X_{-1})$ and $B_2 \in \mathcal L(U_2, X_{-1})$, where $X_{-1}$ is the extrapolation space associated with $A$. Finally, we assume that $F \colon U_1 \times X \to Z$ is bilinear and bounded, i.e. there exists a constant $C>0$ such that
 \begin{equation}\label{eq:F_bdd}
    \norm{F(u_1,x)}_Z \leq C \norm{u_1}_{U_1} \norm{x}_X
 \end{equation}
 holds for all $u_1 \in U_1$ and $x \in X$. 
 Typical examples of $F$ are given by
 \begin{enumerate}[(i)]
     \item $F(u_1,x) = u_1x = \begin{bmatrix} \tilde{u}_1 x & \dots & \tilde{u}_m x\end{bmatrix}^\top$, where $u_1= \begin{bmatrix} \tilde{u}_1 & \dots & \tilde{u}_m \end{bmatrix}^\top \in U_1=\CC^m$ and $Z=X^m$. In this case, we have $B_1= \begin{bmatrix} \tilde B_1 & \dots & \tilde B_m \end{bmatrix}$ with each $\tilde B_i \in \mathcal{L}(X,X_{-1})$ and
     \begin{equation}\label{eq:example_bilinear_sum}
         B_1F(u_1,x) = \sum_{i=1}^m \tilde{u}_i \tilde{B}_i x,
     \end{equation}
     which leads to the prototypical class of bilinear control systems mentioned in the introduction.
     \item $F(u_1,x) = \inprod{u_1}{x}_{X^*,X} z$, where $U_1=X^*$ and $z \in Z$ is a fixed element in some Banach space $Z$.

     \item $F(u_1,x) = \inprod{x^*}{x}_{X^*,X} u_1$, where $U_1=Z$ and $x^* \in X^*$ is fixed.
 \end{enumerate}
The arguments below also cover continuous sesquilinear mappings $F$.
Thus, if $X$ is a Hilbert space, the inner product may be used instead of the
dual pairing in the examples above. For simplicity, we present the results for
bilinear mappings.

Note that $B_1$ and $B_2$ may be unbounded with respect to the $X$-norm, which is typically the case if \eqref{eq:bilinear} arises from a boundary control system, cf.\ \cite[Ch.~10]{MR2502023} for linear systems.
In this sense, we refer to $B_1$ and $B_2$ as unbounded, and we call them bounded if $B_1 \in \mathcal{L}(Z,X)$ and $B_2 \in \mathcal{L}(U_2,X)$, respectively.

Regarding the existence of (classical) solutions, the unboundedness of $B_1$ and $B_2$ plays a critical role.
Even for linear systems with unbounded control operators, the existence and regularity of solutions are nontrivial issues. They are intimately linked to the notion of admissible control operators, see \cite[Ch.~4]{MR2502023}, which we briefly recall below.

\begin{definition}
   Let $X,Z$ be Banach spaces. We call $B \in \mathcal{L}(Z, X_{-1})$ an \emph{$L^p$-admissible control operator} for a $C_0$-semigroup $(T(t))_{t \geq 0}$ (or just \emph{$L^p$-admissible}) for $p \in [1, \infty)$ if there exists $t > 0$ such that, for all $v \in L^p([0,\infty);Z)$,
    \begin{equation}\label{eq:inputmap}
     \Phi_t v \coloneqq \int_0^t T_{-1}(t-s)Bv(s)  \ds \in X.   
    \end{equation}
\end{definition}

\begin{remark}\label{rem:admissibility}
\begin{enumerate}[(i)]
    \item
    It is well-known, see e.g.~\cite[Prop.~4.2.2~\&~Prop.~4.2.4]{MR2502023} for $p=2$, that if $B$ is $L^p$-admissible, then $\Phi_t \in \mathcal{L}(L^p([0, \infty);Z),X)$ holds for all $t \geq 0$, and the \emph{admissibility constants}
    \begin{equation*}
        K_t \coloneqq \nm[\mathcal{L}(L^p([0, \infty); Z), X)]{\Phi_t}
    \end{equation*}
    are non-decreasing with respect to $t\geq 0$. Moreover, since $\Phi_t v$ only depends on the values of $v$ on $[0,t]$, we have that $\Phi_t v$ is well-defined for $v \in L^p_{\loc}([0,\infty);Z)$ and
    \begin{equation*}
        \norm{\Phi_t v}_X \leq K_t \norm{v}_{L^p([0,t];Z)}. 
    \end{equation*}
\item The map $(t,v) \mapsto \Phi_t v$ is continuous from $[0,\infty) \times L^p_{\loc}([0,\infty);Z)$ to $X$.

\item    Every bounded $B \in \mathcal L(Z,X)$ is $L^p$-admissible for all $p \in [1,\infty)$.
    \end{enumerate}
\end{remark}

We are interested in the following standard solution concepts.

\begin{definition}
    Let $x_0 \in X$, $u_1 \in L_{\loc}^p([0, \infty); U_1)$, and $u_2 \in L_{\loc}^p([0, \infty); U_2)$.
    \begin{enumerate}[\rm (i)]
        \item A function $x \in C([0,\infty);X)$ is called a \emph{(global) mild solution of \eqref{eq:bilinear}} if, for all $t \geq 0$,
    \begin{equation*}
         x(t) = T(t)x_0 + \int_{0}^{t} T_{-1}(t - s)(B_1F(u_1(s),x(s)) + B_2u_2(s)) \ds.
    \end{equation*}

    \item A function $x\in C^1([0,\infty);X)$ is called a \emph{(global) classical solution of \eqref{eq:bilinear}} if $x(0)=x_0$ and, for almost every $t \geq 0$,
    \begin{equation}\label{eq:extrapolated_bilinear}
         \dot x(t) = A_{-1}x(t) + B_1F(u_1(t),x(t)) + B_2u_2(t).
    \end{equation}
    \end{enumerate}
\end{definition}

The abstract system class \eqref{eq:bilinear} is discussed in
\cite{MR4440806} in the context of input-to-state stability under slightly more general assumptions on $F$. The following existence result of mild solutions was shown.

\begin{lemma}{\cite[Lem.~2.8~\&~Rem.~2.10]{MR4440806}}\label{lm:bilinear_mild}
Let $X, U_1, U_2, Z$ be Banach spaces, $A$ be the generator of a $C_0$-semigroup on $X$, $B_1 \in \mathcal L(Z,X_{-1})$ be $L^{p_1}$-admissible and $B_2 \in \mathcal L(U_2, X_{-1})$ be $L^{p_2}$-admissible for some $p_1, p_2\in[1,\infty)$, and let $F\colon U_1 \times X \to Z$ be bilinear and bounded. Then, for every $x_0 \in X, u_1\in L^{p_1}_{\loc}([0, \infty);U_1), u_2 \in L^{p_2}_{\loc}([0, \infty);U_2)$, the system \eqref{eq:bilinear} admits a unique global mild solution.
\end{lemma}

\subsection{Classical solutions}

In the following, we extend Lemma~\ref{lm:bilinear_mild} to classical solutions. We will use the following simple fact, obtained by passing to the equivalent $L^p$-norm given by duality.

\begin{lemma}\label{lem:equiv_Lp_norm}
    Let $I \subset \RR$ be an interval, $f \in L^p(I)$, $p \in [1,\infty)$, and $q$ be the H\"older conjugate to $p$, i.e. $\frac{1}{p} + \frac{1}{q}=1$. Then, for every $\epsilon>0$, there exists $g_\epsilon \in L^q(I)$ with $\norm{g_\epsilon}_{L^q} \leq 1$ such that
    \begin{equation*}
        \norm{f}_{L^p(I)} = \sup_{\norm{g}_{L^q(I)} \leq 1} \norm{fg}_{L^1(I)} \le \norm{fg_\epsilon}_{L^1(I)} + \epsilon. 
    \end{equation*}
\end{lemma}
The next theorem is our main result on classical solutions of bilinear systems.
\begin{theorem}\label{thm:classical_solutions}
Under the assumptions of Lemma~\ref{lm:bilinear_mild}, let
\[
    x_0\in X,\qquad
    u_1\in W^{1,p_1}_{\loc}([0,\infty);U_1),\qquad
    u_2\in W^{1,p_2}_{\loc}([0,\infty);U_2)
\]
satisfy the compatibility condition
\[
    A_{-1}x_0+B_1F(u_1(0),x_0)+B_2u_2(0)\in X .
\]
Then, the unique global mild solution of \eqref{eq:bilinear} is a global
classical solution. In particular, $x\in C^1([0,\infty);X)$ and
\[
    \dot x(t)
    =
    A_{-1}x(t)+B_1F(u_1(t),x(t))+B_2u_2(t)
\]
holds for almost every $t\ge0$ in $X_{-1}$.
\end{theorem}
\begin{proof}
    We prove the claim in several steps. 
    
     \emph{Step 1:} Let $x_0, u_1, u_2$ be as in the theorem and let $w_0 \in X$. 
     We first show that for every $t_0\geq 0$ there exists $t_e>t_0$ such that the system
     \begin{equation}\label{eq:aux_derivative_solution}
        \left\{\begin{aligned}
            \dot{w}(t) &= Aw(t) + B_1F(u_1(t),w(t)) + B_1F\left(\dot{u}_1(t), x_0 + \int_{t_0}^t w(s) \ds\right) + B_2\dot u_2(t) ,\quad t \in (t_0,t_e),\\
        w(t_0)&=w_0
        \end{aligned}\right.
    \end{equation}
    has a unique mild solution $w\in C([t_0, t_e];X)$ on $[t_0, t_e]$, that is, for $t \in [t_0,t_e]$, we have
    \begin{equation}\label{eq:int_formula_aux_prb}
        w(t) = T(t-t_0)w_0 + \int_{t_0}^{t} T_{-1}(t-s)\left(B_1F(u_1(s),w(s)) + B_1F\left(\dot u_1(s),x_0 + \int_{t_0}^s w(r)\d r\right) +B_2 \dot u_2(s)\right)\ds.
    \end{equation}
    
    Moreover, we show that $t_e=t_0+\delta$ can be chosen such that $\delta > 0$ depends only on bounded sets containing $x_0$, $w_0$ and $t_0$, respectively, and not on the particular choices $x_0, w_0$ and $t_0$.

    Let $T>0$ and consider $t_0 \in [0,T]$, and let $r > 0$ such that $\nm[X]{x_0}, \nm[X]{w_0}\leq r$. 
    Let $M\ge1,\omega \ge 0$ such that $\norm{T(t)} \le Me^{\omega t}$ for all $t\ge 0$ and define
  \[
  k \coloneqq  3(r+1)Me^{\omega}.
  \]

    By the absolute continuity of the $L^p$-norm with respect to the Lebesgue measure, there exists a $\delta \in (0, 1)$, independent of $t_0$, such that 
\begin{equation}
    K_{T+1}\max\left\{C\left(\norm{u_1}_{L^{p_1}([t_0, t_0+\delta];U_1)} + \norm{\dot u_1}_{L^{p_1}([t_0, t_0+\delta];U_1)}\right), \norm{\dot u_2}_{L^{p_2}([t_0, t_0+\delta];U_2)}\right\} \leq \frac{1}{4}
\end{equation}    
    holds, where $K_{T+1} \coloneqq \max  \{K_{B_1,T+1}, K_{B_2,T+1}\}$ with $K_{B_1,t}, K_{B_2,t}$ being the admissibility constants of $B_1$ and $B_2$ from Remark~\ref{rem:admissibility}, respectively, and $C$ is the constant from \eqref{eq:F_bdd}.
    Define $t_e \coloneqq t_0+\delta$, 
    \begin{equation*}
        S\coloneqq \{x \in C([t_0, t_e];X) \mid \norm x_{C([t_0, t_e]; X)} \le k\},
    \end{equation*}
    and the map $\mathcal T: S \to S$ by defining $\mathcal T(w)(t)$ as the right-hand side of \eqref{eq:int_formula_aux_prb}.
    
    Note that $\mathcal T$ is well-defined. Indeed, the first term on the right-hand side of \eqref{eq:int_formula_aux_prb} is clearly continuous in $t$.
    Since $u_1\in W^{1,{p_1}}_{\loc}([0,\infty);U_1)$, the continuity of $F$ and $w$ implies that $F(\dot u_1,x_0), F(u_1,w), F(\dot u_1, \int_0^{\cdot}w(s)\ds) \in L^{p_1}([t_0,t_e];Z)$. Thus, the remaining terms on the right-hand side of \eqref{eq:int_formula_aux_prb} are continuous with respect to $t$ by the admissibility of $B_1$ and $B_2$, see Remark~\ref{rem:admissibility}.   
    Moreover, for all $t \in [t_0, t_e]$ we have using $k \ge 3, 3r$
    \begin{align*}
        \norm{\mathcal T(w)(t)} 
        &\le Me^{\omega}\norm{w_0}_X + K_{T+1}C \left(\norm{u_1}_{L^{p_1}([t_0, t_e];U_1)} + \norm{\dot u_1}_{L^{p_1}([t_0, t_e];U_1) } \right)\left(\norm{x_0}_X + \norm{w}_{C([t_0,t_e];X)}\right)\\
        & \hspace{1em}+K_{T+1} \norm{\dot u_2}_{L^{p_2}([t_0, t_e];U_2)} \\
        &\le \frac{k}{3} + \frac14\left(\frac k3 + k\right) + \frac k3= k.
    \end{align*}
    Here we used the admissibility inequality for $B_1$ and $B_2$, the boundedness of $F$ and the estimate $\norm{\int_{t_0}^s w(r) \mathrm{d}r} \leq \norm{w}_{C([t_0,t_e];X)}$ for all $s \in [t_0,t_e]$, which follows from $\delta<1$.
    
    We also have that $\mathcal T$ is a contraction as for $w,\tilde w \in S$ a similar estimate as before yields
    \begin{align*}
        \norm{\mathcal T(w) - \mathcal T(\tilde w)} &\le K_{T+1}C \left(\norm {u_1}_{L^{p_1}([t_0, t_e]; U_1)} + \norm{\dot u_1}_{L^{p_1}([t_0, t_e]; U_1)}\right)\norm{w - \tilde w}_{C([t_0,t_e];X)} \\
        &\leq \frac{1}{4} \norm{w - \tilde w}_{C([t_0,t_e];X)}.
    \end{align*}
    By Banach's fixed point theorem, $\mathcal T$ has a unique fixed point, which is the unique mild solution of \eqref{eq:aux_derivative_solution} on $[t_0,t_e]$.

    \emph{Step 2:} We now show that the solution $w$ has a unique global extension.
    Let $ t_{\max} $ be the supremum over all $ t_e > 0 $ such that there exists a unique mild solution $w$ of \eqref{eq:aux_derivative_solution} on $[t_0, t_e]$, where $x_0, w_0 \in X$ and $u_1\in W^{1,{p_1}}_{\loc}([t_0,\infty);U_1), u_2\in W^{1,{p_2}}_{\loc}([t_0,\infty);U_2)$ are given.
    Suppose that $ t_{\max} $ is finite. We will show that $\limsup_{t \to t_{\max}} \|w(t)\|_X = \infty$. If this is not the case, we have
    \[
    r \coloneqq \sup_{t \in [t_0, t_{\max})} \|w(t)\|_X < \infty.
    \]
    Let $(t_n)_{n \in \mathbb{N}}$ be a sequence of positive real numbers converging to $ t_{\max} $ from below. Since $ t_n \in [t_0, t_{\max}] $ and $\|w(t_n)\| \leq r $ for all $ n \in \mathbb{N} $, Step 1 guarantees the existence of some $\delta > 0 $ independent of $x_0, w_0$ and $n \in \mathbb{N}$ 
    such that the system
    \begin{equation*}
        \begin{cases}
            \dot{z}(t) = Az(t) + B_1F(u_1(t),z(t)) +B_1F\left(\dot u_1(t), x(t_n) + \int_{t_n}^t z(s)\ds\right) +B_2\dot u_2(t), \\
            z(t_n) = w(t_n)
        \end{cases}
    \end{equation*}
    has a unique mild solution $ z $ on $[t_n, t_n + \delta]$. Therefore, for $ n \in \mathbb{N} $ with $ t_n + \delta > t_{\max} $, we can extend $ w $ by $ w(t) = z(t) $, $ t \in [t_n, t_n + \delta] $, to a solution of \eqref{eq:aux_derivative_solution} on $[t_0, t_n + \delta]$. This contradicts the maximality of $ t_{\max} $, and hence, $ w $ has to be unbounded in $ t_{\max} $.

    Now, to conclude that $t_{\rm max}=\infty$, it suffices to show that the solution $w$ of \eqref{eq:aux_derivative_solution} remains bounded on every compact subinterval $[t_0, t_e]$ of its maximal interval of existence. Let $t\in [t_0,t_e]$, then we have for any $\epsilon>0$ that
    \begin{align*}
        \norm{w(t)}_X
        &\le Me^{\omega(t-t_0)} \norm{w_0}_X + K_t\norm{\dot u_2}_{L^{p_2}([t_0,t];U_2)}\\
        &\hspace{1em}+ K_t C \left(\norm{ \norm{u_1}_{U_1} \norm{w}_X}_{L^{p_1}([t_0,t])} + \norm{\norm{\dot u_1}_{U_1} \left(\norm{x_0}_X + \int_{t_0}^\cdot \norm{w(s)}_X \ds\right)}_{L^{p_1}([t_0,t])} \right)\\
        &\leq \underbrace{Me^{\omega(t_e-t_0)} \norm{w_0}_X + K_{t_e}C\left(\norm{\dot u_1}_{L^{p_1}([t_0,t_e];U_1)}\norm{x_0}_X + 2\epsilon\right) + K_{t_e}\norm{\dot u_2}_{L^{p_2}([t_0,t_e];U_2)}}_{\eqcolon a} \\
        &\hspace{1em}+ K_{t_e}C \left(\norm{\norm{u_1}_{U_1} g_\epsilon \norm{w}_X}_{L^1([t_0,t_e])} + \norm{ \norm{\dot u_1}_{U_1} h_\epsilon \int_{t_0}^\cdot \norm{w(s)}_X \ds}_{L^1([t_0,t_e])} \right)\\
        &= a + K_{t_e}C \left(\norm{\norm{u_1}_{U_1} g_\epsilon \norm{w}_X}_{L^1([t_0,t_e])} + \norm{\norm{w}_X\int^{t_e}_\cdot  \norm{\dot u_1(s)}_{U_1} h_\epsilon(s) \ds}_{L^1([t_0,t_e])} \right),
    \end{align*}
    where $q_1$ is the Hölder conjugate to $p_1$  and $g_\epsilon, h_\epsilon \in L^{q_1}([t_0, t_e])$ with $\norm{g_\epsilon}_{L^{q_1}([t_0, t_e])},$ $\norm{h_\epsilon}_{L^{q_1}([t_0, t_e])} \le 1$ are chosen according to Lemma~\ref{lem:equiv_Lp_norm}. Moreover, we applied Fubini's theorem to obtain the last term.
 
    By Gronwall's inequality, we obtain
    \begin{align*}
        \norm{w(t)} 
        &\le a \exp\left(K_{t_e}C \norm{u_1g_\epsilon}_{L^1([t_0, t_e];U_1)} + K_{t_e}C\norm{\int_{\cdot}^{t_e} \norm{\dot u_1(s) h_\epsilon(s)}_{U_1} \ds}_{L^1([t_0,t_e])} \right)\\
        &\le a \exp\left(K_{t_e}C \norm{u_1g_\epsilon}_{L^1([t_0,t_e];U_1)} + K_{t_e}C(t_e-t_0)\norm{\dot u_1 h_\epsilon}_{L^1([t_0,t_e];U_1)}\right)\\
        &\le a \exp\left(K_{t_e}C\norm{u_1}_{L^{p_1}([t_0,t_e];U_1)} + K_{t_e}C(t_e-t_0)\norm{\dot u_1}_{L^{p_1}([t_0,t_e];U_1)}\right). 
    \end{align*}
    
    \emph{Step 3:} Let $x_0, u_1, u_2$ be as in the theorem and let $x$ be the associated mild solution of \eqref{eq:bilinear}. Let $w \in C([0,\infty);X)$ be the solution of \eqref{eq:aux_derivative_solution} for $w(0)=w_0 \coloneqq A_{-1}x_0 + B_1F(u_1(0),x_0) + B_2u_2(0) \in X$. We will show that $v \coloneqq x_0 + \int_{0}^{\cdot} w(s) \d s \in C^1([0, \infty);X)$ is a mild solution of \eqref{eq:bilinear}.
    Since, for all $z \in X$, we have $\int_a^b T_{-1}(\tau) z \d \tau \in X$ and
    \begin{equation}\label{eq:aux_sg_int}
        A_{-1} \int_r^t T_{-1}(s-r)z \d s = T(t-r)z-z,
    \end{equation}
    we obtain
    \begin{equation}\label{eq:identity_start_v}
        \begin{aligned}
            x_0 + \int_{0}^t T(s)w_0\ds 
            &= x_0 + \int_{0}^t A_{-1}T_{-1}(s)x_0 \ds + \int_{0}^t T_{-1}(s) \left(B_1F(u_1(0),x_0) + B_2u_2(0)\right) \ds\\
            &= T(t)x_0 + \int_{0}^t T_{-1}(s) \left(B_1F(u_1(0),x_0) + B_2u_2(0) \right)\ds.
        \end{aligned}
    \end{equation}
    Inserting the representation of $w$ given by \eqref{eq:int_formula_aux_prb} into the definition of $v$ yields
    \begin{equation*}
        v(t) = x_0 \!+\! \int_{0}^t\! \left(T(s)w_0 +  \int_{0}^s T_{-1}(s-r)\left(B_1F(u_1(r),w(r)) + B_1F(\dot u_1(r),v(r)) + B_2 \dot u_2(r)\right)\!\!\;\d r\!\right)\! \ds.
    \end{equation*}
    Since $F$ is bilinear and bounded, it is Fr\'echet differentiable with Fr\'echet derivative
    \begin{equation*}
        DF(u_1,x)(h_{u_1},h_x) = F(h_{u_1},x) + F(u_1,h_x)
    \end{equation*}
    and the chain rule implies
    \begin{equation*}
    \frac{\mathrm{d}}{\mathrm{d}r} F(u_1(r),v(r)) = F(u_1(r), w(r)) + F(\dot u_1(r), v(r)),
    \end{equation*}
    where we used $\dot v = w$.
    Thus, \eqref{eq:identity_start_v}, integration by parts, Fubini's theorem, and \eqref{eq:aux_sg_int} yields
    \begin{align}\label{eq:x_0+y(t)}
        v(t) &= T(t)x_0 + \int_{0}^t \left(T_{-1}(s) \left(B_1F(u_1(0),x_0) + B_2u_2(0) \right) + \left[T_{-1}(s-r)\left(B_1F(u_1(r),v(r)) + B_2u_2(r)\right)\right]^{r=s}_{r=0} \right. \notag \\
        &\hspace{1em} \left.+ A_{-1}\int_{0}^s T_{-1}(s-r) \left(B_1F(u_1(r),v(r)) + B_2 u_2(r) \right) \d r\right) \mathrm{d}s\notag\\
        &= T(t)x_0 + \int_{0}^t \left(B_1F(u_1(s),v(s)) + B_2u_2(s)\right) \ds\nonumber \\
        &\hspace{1em}+ \int_{0}^t A_{-1}\int_r^t T_{-1}(s-r) \left(B_1F(u_1(r),v(r)) + B_2u_2(r)\right) \ds \d r \notag\\
        &= T(t)x_0 + \int_{0}^t T_{-1}(t-r)\left(B_1F(u_1(r),v(r)) + B_2u_2(r)\right) \d r.
    \end{align}

    This shows that $v$ is the mild solution of \eqref{eq:bilinear}, and by uniqueness, $x = v \in C^1([0, \infty); X)$.
    
    \emph{Step 4:} Finally, differentiating \eqref{eq:x_0+y(t)} yields
    \begin{align*}
        \dot{x}(t) = \dot{v}(t)
        &= A_{-1}\left(T(t)x_0 + \int_{0}^t T_{-1}(t-s)\left(B_1F(u_1(s),v(s)) + B_2u_2(s)\right) \ds\right) \\
        &\hspace{1em}+ B_1F(u_1(t),v(t)) + B_2u_2(t)\\
        &= A_{-1}x(t) + B_1F(u_1(t),x(t)) + B_2u_2(t)
    \end{align*}
    with differentiation in $X$, since the left-hand side is differentiated in $X$. Thus, $x$ is indeed the classical solution to \eqref{eq:bilinear}.
\end{proof}

\begin{remark}
    Consider the situation around \eqref{eq:example_bilinear_sum}.
    It is evident that $B=\begin{bmatrix} \tilde{B}_1 & \dots & \tilde{B}_m \end{bmatrix}$ is $L^p$-admissible if and only if each $\tilde{B}_i$ is $L^p$-admissible. Furthermore, if each $\tilde{B}_i$ is $L^{p_i}$-admissible, then $B$ is $L^p$-admissible for $p = \max_{i =1,\dots,m}p_i$. Thus, Lemma \ref{lm:bilinear_mild} and Theorem \ref{thm:classical_solutions} apply for this $p$, where these results implicitly assume that each $\tilde{u}_i$ belongs to the same $L^p$ or Sobolev space.
    The latter can be relaxed to $\tilde{u}_i \in L_{\loc}^{p_i}([0,\infty))$ for mild solutions and $\tilde u_i \in W_{\loc}^{1,p_i}([0,\infty))$ for classical solutions. Indeed, the structure of \eqref{eq:example_bilinear_sum} allows us to estimate each integral over $T(t-\cdot) \tilde{u}_i \tilde{B}_ix$ separately using $L^{p_i}$-admissibility of $\tilde{B}_i$. Note that this argument also applies to the proof of Lemma~\ref{lm:bilinear_mild} in \cite{MR4440806}.
\end{remark}

\begin{remark}
    For linear control systems, i.e.  for $B_1=0$, one also obtains classical solutions under the regularity condition $u_2 \in W^{2,1}_{\loc}([0,\infty))$, provided that $A_{-1}x_0 + B_2u_2(0) \in X$, without requiring the admissibility of $B_2$, see e.g. \cite[Thm.~3.8.3]{staffans2005well}.
    However, in the bilinear case the situation is more delicate. In general, arbitrarily high regularity of $u_1$ alone does not ensure even the existence of mild solutions. Indeed, let $A$ be a generator of a $C_0$-semigroup, $B_1=A_{-1}$, and $F(u_1,x)=u_1x$ for $u_1 \in \CC$. Choosing $u_1=-2$ in a neighborhood of $0$ and $u_2=0$ leads to the backward Cauchy problem which is in general not well-posed. 
\end{remark}

\subsection{Continuous dependence on the data}
In this subsection, we prove that perturbations on the initial data $x_0, u_1, u_2$ only continuously change the solution $x$ on finite intervals. We first prove an a~priori norm estimate on the state $x$ on a finite interval.

\begin{lemma}\label{lem:x_bdd_by_data}
    Let $p_1,p_2 \in [1,\infty)$ and consider the bilinear system \eqref{eq:bilinear} with $L^{p_1}$-admissible, and $L^{p_2}$-admissible control operators $B_1$, and $ B_2$. Let $M \ge 1,\omega \geq 0$ such that $\norm{T(t)} \leq Me^{\omega t}$ holds for all $t \geq 0$. Then, for every $x_0 \in X$, $u_1 \in L^{p_1}_{\loc}([0,\infty); U_1)$, and $u_2 \in L^{p_2}_{\loc}([0,\infty); U_2)$ the mild solution $x$ of \eqref{eq:bilinear} satisfies for all $t \in [0, t_e]$
\begin{equation}\label{est:x_bdd_by_data}
    \norm{x(t)}_X \leq \left(Me^{\omega t_e}\norm{x_0}_X + K_{t_e} \norm{u_2}_{L^{p_2}([0,t_e]; U_2)}\right)\exp\left(K_{t_e}C \norm{u_1}_{L^{p_1}([0,t_e];U_1)}\right),
\end{equation}
where $K_{t_e} \coloneqq \max\{K_{B_1, t_e}, K_{B_2, t_e}\}$ is the maximum of the admissibility constants of $B_1$ and $B_2$ at time $t_e$. 
\end{lemma}

\begin{proof}
The mild solution of $\eqref{eq:bilinear}$ satisfies
\begin{equation*}
    x(t)=T(t)x_0 +\int_{0}^t T_{-1}(t-s)(B_1F(u_1(s),x(s)) + B_2u_2(s)) \ds.
\end{equation*}
We deduce from the boundedness of $F$ and the admissibility of $B_1$ and $B_2$ that
\begin{equation*}
    \begin{aligned}
        \nm[X]{x(t)}&\leq Me^{\omega t}\nm[X]{x_0} + K_t C \nm[L^{p_1}({[0, t]})]{\norm{u_1}_{U_1} \norm{x}_X} + K_t \norm{u_2}_{L^{p_2}([0,t]; U_2)}\\
        & \leq \underbrace{Me^{\omega t_e}\nm[X]{x_0} + K_{t_e} \norm{u_2}_{L^{p_2}([0,t_e]; U_2)} +\epsilon}_{ \eqcolon a} + K_{t_e}C\nm[L^1({[0,t_e]})]{\norm{u_1}_{U_1} g_\epsilon \norm{x}_X},
    \end{aligned}
\end{equation*}
where $\epsilon > 0$, $q_1$ is the H\"older conjugate to $p_1$ and $g_\epsilon \in L^{q_1}([0,t_e])$  with $\nm[L^{q_1}{([0,t_e])}]{g_\epsilon} \le 1$ is chosen as in Lemma~\ref{lem:equiv_Lp_norm}. Gronwall's inequality and H\"older's inequality imply
\begin{equation*}
    \norm{x(t)}_X \le a \exp\left(K_{t_e}C \norm{\norm{u_1}_{U_1} g_\epsilon}_{L^1([0,t_e])}\right) \le a\exp\left(K_{t_e}C \norm{u_1}_{L^{p_1}([0,t_e];U_1)}\right).
\end{equation*}
Since this holds for every $\epsilon>0$, we obtain \eqref{est:x_bdd_by_data}.
\end{proof}
We now show the continuous dependence of the solutions on the initial condition $x_0$ and inputs~$u_1$ and $u_2$. 
\begin{proposition}\label{pr:continuous_dependence}
    Let $p_1,p_2 \in [1,\infty)$ and consider the bilinear system \eqref{eq:bilinear} with $L^{p_1}$-admissible, and $L^{p_2}$-admissible control operators $B_1$, and $ B_2$. Then, for any $t_e > 0$ the mild solution depends continuously in $C([0, t_e];X)$ on $x_0 \in X$, $u_1 \in L^{p_1}([0, t_e];U_1)$, and $u_2 \in L^{p_2}([0, t_e];U_2)$.  
\end{proposition}

\begin{proof}
Let $x_0, \tilde x_0 \in X$, $u_1, \tilde u_1 \in L^{p_1}([0,t_e]; U_1)$, and $u_2, \tilde u_2 \in L^{p_2}([0,t_e]; U_2)$. Then, the difference of the corresponding mild solutions satisfies
\begin{align*}
        x(t)-\tilde x(t)&= T(t)(x_0-\tilde x_0)
\\&+\! \int_{0}^t T_{-1}(t-s)\big(B_1F(u_1(s),x(s))-B_1F(\tilde u_1(s),\tilde x(s)) + B_2(u_2(s)-\tilde u_2(s)) \big)\;\!\!\ds.
\end{align*}
Split $F(u_1,x) - F(\tilde u_1,\tilde x) = F(u_1-\tilde u_1, x) + F(\tilde u_1, x - \tilde x)$ and let $K_{t_e} \coloneqq \max\{K_{B_1, t_e}, K_{B_2, t_e}\}$ be the maximum of the admissibility constants of $B_1$ and $B_2$ at time $t_e$. Then, we obtain
\begin{equation*}
    \begin{aligned}
        \MoveEqLeft \nm[X]{x(t)-\tilde x(t)} \\
        &\le Me^{\omega t}\nm[X]{x_0 - \tilde x_0} + K_t\norm{u_2-\tilde u_2}_{L^{p_2}([0,t_e];U_2)} + K_tC \nm[L^{p_1}({[0, t]})]{\norm{u_1- \tilde u_1}_{U_1} \norm{x}_X}\\
        &\hspace{1em} + K_tC \nm[L^{p_1}({[0,t]})]{ \norm{\tilde u_1}_{U_1} \norm{x-\tilde x}_X} \\
        &\le \underbrace{Me^{\omega t_e}\nm[X]{x_0 - \tilde x_0} + K_{t_e}\norm{u_2-\tilde u_2}_{L^{p_2}([0,t_e];U_2)} + K_{t_e}C \sup_{s\in [0,t_e]} \norm{x(s)}_X \norm{u_1-\tilde u_1}_{L^{p_1}([0,t_e]; U_1)} + \epsilon}_{\eqcolon a}\\
        &\hspace{1em} + K_{t_e}C \norm{\norm{\tilde u_1}_{U_1} g_\epsilon \norm{x-\tilde x}_X}_{L^1([0,t_e])},
    \end{aligned}
\end{equation*}
where $\epsilon > 0$ is arbitrary, $q_1$ is the H\"older conjugate to $p_1$ and $g_\epsilon \in L^{q_1}([0, t_e])$ with $\norm{g_\epsilon}_{L^{q_1}([0, t_e])}\le 1$ is chosen as in Lemma~\ref{lem:equiv_Lp_norm}. 
By Lemma \ref{lem:x_bdd_by_data}, $a$ is indeed bounded. Gronwall's inequality and H\"older's inequality imply
\begin{equation*}
    \norm{x(t)-\tilde x(t)}_X \le a \exp\left(K_{t_e}C \norm{\norm{\tilde u_1}_{U_1} g_\epsilon}_{L^1([0,t_e])}\right) \le a \exp\left(K_{t_e}C \norm{\tilde u_1}_{L^{p_1}([0,t_e]; U_1)}\right),
\end{equation*}
which shows continuous dependence on finite intervals.
\end{proof}

\section{Passivity properties of bilinear control systems}
\label{sec:passive}
In this section, we use the solution theory developed in
Section~\ref{sec:well-posedness} to prove passivity for a class of abstract
bilinear systems. Throughout this section, $X$, $U_1$, $U_2$, and $Z$ are
Hilbert spaces. We equip $U_1\times U_2$ with the inner product
\[
\left\langle
\begin{bmatrix}v_1\\v_2\end{bmatrix},
\begin{bmatrix}w_1\\w_2\end{bmatrix}
\right\rangle_{U_1\times U_2}
\coloneqq
\langle v_1,w_1\rangle_{U_1}
+
\langle v_2,w_2\rangle_{U_2}.
\]
For fixed $x\in X$, let $F_x\colon U_1\to Z$ be given by
$F_xu=F(u,x)$. Then $F_x$ is bounded, and the map
$x\mapsto F_x$ is bounded from $X$ to $\mathcal L(U_1,Z)$. To investigate passivity, we consider \eqref{eq:bilinear} with a co-located output
\begin{equation}\label{eq:co-located_mild_bilinear}\tag{$\Sigma_{\text{co}}$}
    \left\{\begin{aligned}
        \dot x(t) &= Ax(t) + B_1F(u_1(t),x(t)) + B_2u_2(t),\qquad  t\in (0, \infty),\\
        x(0) &= x_0,\\
        y(t) &= \begin{bmatrix}F_{x(t)}^* B_1^*x(t) \\ B_2^*x(t)\end{bmatrix} \in U_1 \times U_2,
    \end{aligned}\right.
\end{equation}
where, for the moment, we assume that $B_1$ and $B_2$ are bounded, and where
$F_x^*$ and $B_i^*$, $i=1,2$, denote the adjoint operators of $F_x$ and $B_i$, respectively. Consequently, the mappings $x \mapsto F_x^*B_1^*x$ and $x \mapsto B_2^*x$ are well-defined and continuous from $X$ to $U_1$ and from $X$ to $U_2$, respectively.

If $U_1= \CC$ and $F(u,x)=ux$, then $F_x^*B_1^*x = \inprod{x}{B_1x}_X$ and if $U_1 = \CC^m$, $F(u,x)=ux \in X^m$ and $B_1 = \begin{bmatrix}
    \tilde B_1 & \dots & \tilde B_m
\end{bmatrix}^\top$, then \[
F_x^*B_1^*x = \begin{bmatrix}
    \inprod{x}{\tilde B_1x} & \dots & \inprod{x}{\tilde B_mx}
\end{bmatrix}^\top.
\]

Recall that the boundedness of $B_1$ and $B_2$ implies $L^p$-admissibility for any $p$ and thus there exists a unique mild solution $x \in C([0,\infty);X)$ for every $x_0 \in X$, $u_1 \in L^1_{\loc}([0,\infty);U_1)$ and $u_2 \in L^1_{\loc}([0,\infty);U_2)$.
Moreover, if $x_0 \in X, u_1 \in W^{1,1}_{\loc}([0,\infty);U_1)$, and $u_2 \in W^{1,1}_{\loc}([0,\infty);U_2)$ satisfy $A_{-1}x_0 + B_1F(u_1(0),x_0) + B_2u_2(0) \in X$, then the mild solution is a classical solution and
\[
A_{-1}x(t) = \dot{x}(t) - B_1F(u_1(t),x(t)) - B_2u_2(t) \in X
\]
holds for $t \in [0,\infty)$ since $B_1$ and $B_2$ are bounded on $X$. Thus, we have $A_{-1}x(t) \in X$ which implies $x(t) \in D(A)$ and $A_{-1}x(t)=Ax(t)$. 

We recall the definition of a passive input-output system based on \cite{Staffans2002} tailored to our setting. 
\begin{definition}
System \eqref{eq:co-located_mild_bilinear} is called
\emph{passive} (or \emph{impedance-passive}) if there exists a function
$S\colon X\to[0,\infty)$ such that
\begin{equation}
\label{eq:dissipation}
S(x(t_1))-S(x(t_0))
\le
\Re\int_{t_0}^{t_1}
\left\langle
\begin{bmatrix}u_1(s)\\u_2(s)\end{bmatrix},
y(s)
\right\rangle_{U_1\times U_2}\,\mathrm ds
\end{equation}
holds for all $0\le t_0\le t_1<\infty$ and all admissible inputs for which
the corresponding mild solution exists. Such a function $S$ is called a
\emph{storage function}.
\end{definition}
Recall the Lumer-Phillips theorem, which states that a linear operator $A\colon D(A) \subset X \to X$ generates a contractive $C_0$-semigroup on the Hilbert space $X$ if and only if $A$ is dissipative and $A-\lambda$ is surjective for some $\lambda > 0$. 

\begin{theorem}\label{th:passivity_bounded_B}
    Let $X,U_1,U_2,Z$ be Hilbert spaces, $A$ generate a contractive $C_0$-semigroup on $X$, $B_1\in \mathcal{L}(Z,X)$, and $B_2 \in \mathcal L(U_2, X)$. Then the following assertions hold.
    \begin{enumerate}[(i)]
        \item If $u_1 \in W^{1,1}_{\loc}([0, \infty);U_1)$, $u_2 \in W^{1,1}_{\loc}([0, \infty); U_2)$ and $x_0 \in D(A)$, then the classical solution $x$ of \eqref{eq:co-located_mild_bilinear} satisfies
        \begin{equation*}
            \frac12\frac{\d}{\d t}\norm{x(t)}_X^2 \le \Re \inprod{\begin{bmatrix}u_1(t) \\ u_2(t)\end{bmatrix}}{y(t)}_{U_1 \times U_2} \qquad \text{for $t \in [0, \infty)$}.
        \end{equation*}

        \item If $u_1 \in L^{1}_{\loc}([0, \infty);U_1)$, $u_2 \in L^{1}_{\loc}([0, \infty); U_2)$ and $x_0 \in X$, then the mild solution $x$ of \eqref{eq:co-located_mild_bilinear} satisfies
        \begin{equation*}
            \frac12\norm{x(t)}_X^2 - \frac12\norm{x_0}_X^2 \le \Re \int_0^t \inprod{\begin{bmatrix}u_1(s) \\ u_2(s)\end{bmatrix}}{y(s)}_{U_1 \times U_2} \ds \qquad \text{for $t \in [0, \infty)$}.
        \end{equation*}
    \end{enumerate}
\end{theorem}
\begin{proof}
    For (i), since $A$ is dissipative, every classical solution satisfies, for $t>0$,
    \begin{equation*}
    \begin{aligned}
        \frac12\frac{\d}{\d t}\norm{x(t)}_X^2 = \Re \inprod{\dot x(t)}{x(t)}_X &= \Re \inprod{Ax(t)}{x(t)}_X + \Re\inprod{B_1F(u_1(t),x(t)) + B_2u_2(t)}{x(t)}_X\\
        &\le \Re \inprod{\begin{bmatrix}u_1(t) \\ u_2(t)\end{bmatrix}}{y(t)}_{U_1 \times U_2}.
    \end{aligned}
    \end{equation*}

    For (ii), by integrating the inequality in (i), we obtain the claim for classical solutions immediately. Since $D(A)$ is dense in $X$ and $ W^{1, 1}_{\loc}([0, \infty);U_i)$ is dense in $L^{1}_{\loc}([0, \infty);U_i)$ for $i =1,2$, it follows from Theorem~\ref{thm:classical_solutions} and Proposition~\ref{pr:continuous_dependence} that every mild solution can be approximated uniformly on compact intervals in the $X$-norm by classical solutions. Finally, since $y = \begin{bmatrix}F_x^* B_1^*x \\ B_2^*x\end{bmatrix}$ depends continuously on $x$ the claim follows by approximation.
\end{proof}

Now suppose that $B_1$ and $B_2$ are unbounded. While $\dot{x}(t) = A_{-1}x(t) + B_1F(u_1(t),x(t)) + B_2u_2(t) \in X$ still holds for classical solutions, the individual terms $A_{-1}x(t)$, $B_1F(u_1(t),x(t))$, and $B_2u_2(t)$ do not have to be in $X$ but only in $X_{-1}$ when separated. Thus, splitting $\inprod{\dot x}{x}_X$ as done in the previous proof is not feasible anymore.
To overcome this issue one may assume additional regularity of $A$, $B_1$ and $B_2$.

Let $A$ be self-adjoint and negative, that is $\langle Ax,x\rangle \leq 0$ for all $x \in D(A)$, so that $A$ generates a bounded analytic semigroup. Let us briefly recall the standard construction of the spaces $X_{\frac{1}{2}}$ and $X_{-\frac{1}{2}}$ associated to $A$. For $\lambda>0$ the operator $\lambda-A$ is strictly positive, thus
\begin{equation*}
    \begin{aligned}
        \norm x^2_{X_{\frac12}} &= \inprod{(\lambda-A)x}{x}_X \qquad &\text{ for } x \in D(A),\\
        \norm x_{X_{-\frac12}} &= \sup_{\norm v_{X_{\frac12}}\le 1}\left|\inprod{x}{v}_X\right| \qquad &\text{ for } x \in X
    \end{aligned}
\end{equation*}
define norms on $D(A)$ and $X$, respectively. Note that different choices of $\lambda>0$ lead to equivalent norms.
The fractional interpolation space $X_{\frac12}$ is defined as the completion
of $D(A)$ with respect to $\|\cdot\|_{X_{\frac12}}$. The fractional
extrapolation space $X_{-\frac12}$ is the dual space of $X_{\frac12}$ with
respect to the pivot space $X$; equivalently, it is obtained as the completion
of $X$ with respect to $\|\cdot\|_{X_{-\frac12}}$, cf.~\cite[Sec.~2.10]{MR2502023}.
The inner product $\langle\cdot,\cdot\rangle_X$ continuously extends to a
sesquilinear duality pairing
\[
    \langle\cdot,\cdot\rangle_{X_{-\frac12},X_{\frac12}}
\]
between $X_{-\frac12}$ and $X_{\frac12}$. Moreover,
$A_{-1}\colon X_{\frac12}\to X_{-\frac12}$ is bounded and satisfies
\[
    \Re\langle A_{-1}x,x\rangle_{X_{-\frac12},X_{\frac12}}\le0,
    \qquad \text{for }x\in X_{\frac12}.
\]

\begin{theorem}\label{th:passivity_unbounded}
    Let $X, U_1, U_2, Z$ be Hilbert spaces, $A$ self-adjoint and negative, $B_1\in \mathcal{L}(Z,X_{-\frac{1}{2}})$, and $B_2\in \mathcal{L}(U_2,X_{-\frac12})$. Then, for all $x_0 \in X$, $u_1\in L^{2}_{\loc}([0,\infty);U_1)$, and $u_2\in L^{2}_{\loc}([0,\infty);U_2)$, the corresponding mild solution $x$ satisfies 
    \begin{equation*}
        x \in H^1_{\loc}((0,\infty); X_{-\frac{1}{2}}) \cap C([0, \infty); X) \cap L^2_{\loc}((0, \infty); X_{\frac{1}{2}}),
    \end{equation*}
and the output $y= \left[\begin{smallmatrix} F_{x(\cdot)}^*B_1^*x(\cdot)\\B_2^*x(\cdot)\end{smallmatrix}\right]$ of \eqref{eq:co-located_mild_bilinear} is well-defined as function in $L^2_{\rm loc}([0,\infty);U_1 \times U_2)$. Moreover, for all $ t \in [0, \infty) $, we have
\begin{align*}
\frac12\|x(t)\|_X^2 - \frac12\|x_0\|_X^2
&= \Re\int_0^t \inprod{Ax(s)}{x(s)}_{X_{-\frac12},X_{\frac12}} + \inprod{B_1F(u_1(s),x(s)) + B_2u_2(s)}{x(s)}_{X_{-\frac12},X_{\frac12}} \ds \\ 
&\leq \Re\int_0^t  \inprod{B_1F(u_1(s),x(s)) + B_2u_2(s)}{x(s)}_{X_{-\frac12},X_{\frac12}} \ds\\
&= \Re\int_0^t \left\langle \begin{bmatrix}u_1(s)\\u_2(s)\end{bmatrix}, y(s) \right\rangle_{U_1\times U_2} \ds.
\end{align*}
\end{theorem}
\begin{proof}
    Since $B_1 \in \mathcal{L}(Z,X_{-\frac{1}{2}})$ and $B_2 \in \mathcal{L}(U_2,X_{-\frac{1}{2}})$, they are $L^2$-admissible with respect to $(T(t))_{t\ge0}$, see e.g.~\cite[Prop.~6.5]{MR3230878}. According to Lemma \ref{lm:bilinear_mild}, we have a global mild solution $x$ and $v \coloneqq F(u_1,x) \in L^2_{\loc}([0, \infty); Z)$. Note that $x$ is also the mild solution of the linear control problem
    \begin{equation*}
        \left\{\begin{aligned}
        \dot x(t) &= Ax(t) + [B_1 \;\; B_2] \begin{bmatrix}
            v(t) \\ u_2(t)
        \end{bmatrix}, \qquad t\in (0, \infty),\\
        x(0) &= x_0,
        \end{aligned}\right.
    \end{equation*}
    with the augmented control input $\begin{smallbmatrix}
        v\\ u_2
    \end{smallbmatrix}$.    
    The claimed regularity of $x$ now follows analogously to Lemma 3.6 in \cite{hosfeld2024input}, which requires the continuous dependence on the initial and input data as derived in Proposition~\ref{pr:continuous_dependence}. In particular, the boundedness assumption on $F$ and $B_1$ imply $F_{x(\cdot)}^* \in C([0,\infty);\mathcal{L}(Z,U_1))$ as well as $B_1^*x(\cdot) \in L^2_{\loc}([0,\infty);Z)$. Thus, the first component of $y$ has the claimed regularity. The regularity of the second component is clear by the assumption on $B_2$.
    Finally, the identities and the inequality in the theorem follow as in Lemma 3.6 in \cite{hosfeld2024input}.
\end{proof}

\section{Verification of passivity for various examples}
\label{sec:applications}
\subsection{Controlled bilinear Schrödinger equation}
The wave function of a particle in one spatial dimension $\psi:\RR\rightarrow\CC$ is described by the Schrödinger equation. By applying external electric fields $\tilde u_i\in L^2(\RR)$, $i=1,\ldots,m$, we obtain the bilinear controlled version of the Schrödinger equation \cite{ChamP23} 
\begin{align}
\label{eq:bilinear_schroedinger}
\left\{
\begin{aligned}
&\mathrm{i}\partial_t\psi(t,\xi)=\left(-\partial_\xi^2+V(\xi)+\sum_{i=1}^m\tilde u_i(t)W_i(\xi)\right)\psi(t,\xi),\quad  (t,\xi) \in (0, \infty) \times \RR,\\
&\psi(0,\cdot)=\psi_0\in L^2(\RR;\CC)
\end{aligned}
\right.
\end{align}
with potentials $W_i\in L^\infty(\RR;\CC)$, $i=1,\ldots,m$, and $\partial_\xi^2$ is the closure of the second derivative on the functions with compact support leading to a self-adjoint operator in the Hilbert space $X=L^2(\RR;\CC)$. 

Furthermore, in many applications one can apply perturbation results to show that $-\partial_\xi^2+V$ is self-adjoint, which is trivial for bounded perturbations $V\in L^\infty(\RR)$, but also remains true for potentials that are not square integrable which arise for example in a~quantum harmonic oscillator \cite{Beauchard2013}.  Since  $-\partial_\xi^2+V$ is self-adjoint it follows that the operator
$A\coloneqq \mathrm{i}\partial_\xi^2-\mathrm{i}V$ is skew-adjoint and therefore it generates a unitary group. In particular, the semigroup generated by $A$ is contractive. Thus, the Schrödinger equation \eqref{eq:bilinear_schroedinger} can be written as a abstract bilinear control system of the form \eqref{eq:bilinear} with bounded multiplication operators $\tilde{B}_i=-\mathrm{i}W_i(\cdot)\in \mathcal{L}(X)$, $i=1,\ldots,m$, and bilinear mapping $F$ given by \eqref{eq:example_bilinear_sum}. If we add the co-located output $y$ as in \eqref{eq:co-located_mild_bilinear}, the Schrödinger equation is passive according to Theorem~\ref{th:passivity_bounded_B} with storage function $S(x) = \frac{1}{2}\norm{x}_X^2$.

\subsection{Fokker-Planck equation}
\label{ssec:fokker}
    Let $\Omega \subset \RR^n$ be a bounded domain with $C^2$-boundary $\partial\Omega$ and denote by $\vec n(\xi)$ the outward-pointing unit normal vector at $\xi \in \partial\Omega$. Following \cite{Breiten18}, we consider the bilinear controlled Fokker--Planck equation
    \begin{equation}\label{eq:fokker_pl}
        \left\{
        \begin{aligned}
            \partial_t x(t,\xi) &= \nu \Delta x(t,\xi) + \mathrm{div}(x(t,\xi) \nabla W(\xi)) + u(t)\mathrm{div}(x(t,\xi)\nabla \alpha(\xi)),\quad (t,\xi) \in (0,\infty) \times \Omega\\
            x(0,\cdot) &= x_0 \in L^2(\Omega)
        \end{aligned}
        \right.
    \end{equation}
    with the reflecting boundary condition
    \begin{equation*}
        (\nu \nabla x + x \nabla W) \vec{n} =0 \quad \text{on } \partial \Omega.
    \end{equation*}
    Here, $\Delta, \mathrm{div}, \nabla$ act on the spatial component $\xi$, $\nu > 0$ is constant, and $\alpha, W \in W^{2,p}(\Omega) \cap W^{1,\infty}(\Omega)$, where $p = 2$ if $n=1$, $p>2$ if $n=2$, and $p = n$ in all higher dimensions. Note that the intersection with $W^{1,\infty}(\Omega)$ is redundant for $n=1,2$ due to the corresponding Sobolev embedding. Moreover, we assume $\nabla\alpha \cdot \vec n = 0$ on $\partial\Omega$ in the weak sense.

    The Fokker--Planck equation models the probability density related to the motion of a large set of dragged Brownian particles in the domain $\Omega$ which are reflected at the boundary and which can be manipulated on a microscopic level by means of optical tweezers. In this context, $x_0$ is assumed to be the initial probability density, i.e. $\int_\Omega x_0(\xi) \,\mathrm{d} \xi = 1$, which is not necessary for our purposes.

    The Fokker--Planck equation can be written as a bilinear control system \eqref{eq:bilinear}
    on the state space $X= L^2(\Omega)$ with $A: D(A) \subset X \to X$ given by
    \begin{equation*}
        \begin{aligned}
            Ax &\coloneqq \nu \Delta x+\operatorname{div}(x \nabla W),\quad D(A) \coloneqq \left\{x \in \mathrm{H}^2(\Omega) \mid(\nu \nabla x+x \nabla W) \cdot \vec{n}=0 \text { on } \partial \Omega\right\},
        \end{aligned}
    \end{equation*}
    $B_1 \in \mathcal{L}(L^2(\Omega),X_{-1})$ being the unique extension of
    \begin{equation*}
        B \colon H^1(\Omega) \to L^2(\Omega), \quad Bx \coloneqq \mathrm{div}(x\nabla \alpha),
    \end{equation*}
    $B_2=0$, and $F\colon \CC \times X \to X$, $F(u_1,x) = u_1 x$. 
 
    As shown in \cite{Breiten18}, the operator $A$ has a simple eigenvalue at $0$ with eigenfunction given by $e^{\frac\Phi2}$, where $\Phi(\xi) = \log(\nu) + \frac{W(\xi)}{\nu}$. Let $M$ be the multiplication operator multiplying with $e^{\frac\Phi2}$. In \cite[Ch.~3]{MR4440806}, alternatively in \cite[Sec.~5.2]{hosfeld25}, it is proven, among many other facts, that $M$ is an isomorphism on $L^2(\Omega)$, $H^1(\Omega)$ and $H^2(\Omega)$, respectively, $\hat A \coloneqq MAM^{-1}$ with domain $D(\hat A) = MD(A)$ is self-adjoint and negative, and that $\hat B_1 \coloneqq MB_1M^{-1}$ is a bounded operator from $X$ to $\hat X_{-1/2}$. Here, $\hat X_{1/2}$ and $\hat X_{-1/2}$ are the fractional inter- and extrapolation space associated to the strictly negative operator $\hat{A}- I$ on $\hat{X}=X$ equipped with the norm $\norm{\cdot}_{\hat X} \coloneqq \norm{M\cdot}_{X}$. Note that $\norm{\cdot}_X$ and $\norm{\cdot}_{\hat X}$ as well as $\norm{\cdot}_{\hat X_\alpha}$ and $\norm{\cdot}_{X_\alpha}$ are equivalent, since $M$ is an isomorphism and by the definition of $\hat{A}$ and the fractional norms. By \cite[Prop.~5.1.3.]{MR2502023}, $\hat B_1$ is $L^2$-admissible for $\hat A$. Thus, Theorem~\ref{th:passivity_unbounded} is applicable for $\hat A$ and $\hat B_1$ when using  the co-located output 
    \begin{equation}
    \label{eq:co-located_fokker_pl}
        y(t) = \inprod{\hat B_1x(t)}{x(t)}_{\hat X_{-\frac12}, \hat X_{\frac12}},
    \end{equation}
    which gives the passivity of the Fokker-Planck system~\eqref{eq:fokker_pl} and \eqref{eq:co-located_fokker_pl} with respect to the equivalent storage function $S(x) = \frac12 \norm{x}_{\hat X}^2$ 
    \begin{equation*}
        \begin{aligned}
            \frac12\norm{x(t)}_{\hat X}^2 - \frac12\norm{x_0}_{\hat X}^2
= \Re\int_0^t \inprod{\hat Ax(s)}{x(s)}_{\hat X_{-\frac12}, \hat X_{\frac12}} + \inprod{u_1(s)\hat B_1x(s)}{ x(s)}_{\hat X_{-\frac12}, \hat X_{\frac12}} \ds 
\le \Re \int_0^t u(s)y(s) \ds.
        \end{aligned}
    \end{equation*}

\subsection{District heating pipe}
The evolution of the temperature of water
$\theta:[0,\infty)\times[0,\ell]\to\RR$ in a district heating pipe of
length $\ell>0$ can be derived from an energy balance
\cite{vdHeFRSSBMN17}. If the flow velocity
$u\in L^1_{\loc}([0,\infty);\RR)$ is assumed to be constant across the
spatial domain, the model is given by the bilinear equation
\begin{align}
\label{eq:bilinear_heat_pipe}
\left\{
\begin{aligned}
&        \partial_t \theta(t,\xi) = \kappa\partial_\xi^2 \theta(t,\xi) - u(t) \partial_\xi \!\!\;\theta(t,\xi) \;-\alpha (\theta(t,\xi)-\theta_{\rm ref}(t)), \quad (t,\xi) \in (0, \infty) \times [0,\ell],\\ &\theta(0,\cdot)\in L^2([0,\ell])
\end{aligned}
\right.
\end{align}
where $\kappa$ is a conductivity parameter, the term $\partial_\xi^2 \theta$ models the axial heat diffusion and the advection term $u(t) \partial_\xi \!\!\;\theta(t,\xi)$ describes the convective heat transport along the spatial domain $[0,\ell]$, and  $\theta_{\rm ref}:[0,\infty)\rightarrow\RR$ with $\theta_{\rm ref}\in L^\infty([0,\infty))$ is a given ambient temperature, $\alpha>0$ describes the heat loss. Note that for simplicity, we set all remaining physical  parameters such as the fluid density and pipe diameter to one, as this does not change the admissibility and passivity properties of the system. The equation~\eqref{eq:bilinear_heat_pipe} can be cast as an abstract bilinear system of the form~\eqref{eq:bilinear} in the Hilbert space $X=L^2([0,\ell])$ and with the operator $A=\kappa\partial_\xi^2-\alpha$ with suitable boundary conditions which can be, for instance, 
\begin{align}
    \label{eq:domA_pipe}
D(A)=\{x\in H^2([0,\ell])~|~ x(0)=0,\quad x'(\ell)=0\}.
\end{align}
The Dirichlet boundary condition in \eqref{eq:domA_pipe} at $\xi=0$ can be viewed as a prescribed input temperature value. 
The Neumann boundary condition at $\xi=\ell$ ensures that there is no axial heat conduction across the outlet, i.e.\ the heat energy leaves only by the advection. 

It is known that $A$ with domain \eqref{eq:domA_pipe} is self-adjoint and strictly negative. Furthermore, we set $B_1 \coloneqq \partial_\xi$ and $B_2\coloneqq \alpha\theta_{\rm ref}$ viewed as a multiplication operator which is bounded on $L^2$.

Similar to the treatment of the Fokker-Planck equation in Subsection~\ref{ssec:fokker}, we can show that $B_1\in\mathcal{L}(X,X_{-\frac12})$ which gives the $L^2$-admissibility. Furthermore, $B_2$ is bounded and therefore also $L^2$-admissible. Therefore, we obtain the passivity of the system \eqref{eq:bilinear_heat_pipe} from Theorem~\ref{th:passivity_unbounded}.

Alternatively to the boundary conditions presented in \eqref{eq:domA_pipe}, a negative operator $A$ can also be obtained by imposing Neumann boundary conditions or periodic boundary conditions of the form $\theta(t,0)=\theta(t,\ell)$ for all $t\geq 0$, i.e.\ the water in the pipe is running in a loop leading to analogous passivity results. 

\section{Conclusion}
\label{sec:conclusion}

We studied a general class of abstract bilinear control systems in Banach
spaces and proved existence and uniqueness of classical solutions under
admissibility assumptions. Moreover, we established continuous dependence of
mild solutions on the initial and input data. This is utilized to prove passivity for our abstract system class when the underlying space is a  Hilbert space. The admissibility assumptions leading to passivity are verified for various examples including a bilinear Schrödinger equation, a~Fokker-Planck equation and a fluid running in a heating or cooling pipe.
Possible future work may include an investigation of stabilizability for our class of bilinear control systems and the study of passivity notions in Banach spaces.

\bibliographystyle{abbrv}
\bibliography{references}
%%-----------------------------
%%      your bibliography
%%-----------------------------
\end{document}